\documentclass[]{amsart}
\usepackage[all]{xy}
\usepackage{amscd,amsthm,amssymb,amsfonts,amsmath,euscript}

\theoremstyle{plain}
\newtheorem{thm}{Theorem}[section]
\newtheorem{lemma}[thm]{Lemma}
\newtheorem{prop}[thm]{Proposition}
\newtheorem{cor}[thm]{Corollary}

\theoremstyle{definition}
\newtheorem{defn}[thm]{Definition}

\theoremstyle{remark}
\newtheorem{remark}[thm]{Remark}
\newtheorem*{thank}{{\bf Acknowledgments}}
\newcommand{\nc}{\newcommand}

\def\makeop#1{\expandafter\def\csname#1\endcsname
  {\mathop{\rm #1}\nolimits}\ignorespaces}
\makeop{Hom}   \makeop{End}   \makeop{Aut}   \makeop{Isom}  \makeop{Pic} 
\makeop{Gal}   \makeop{ord}   \makeop{Char}  \makeop{Div}   \makeop{Lie} 
\makeop{PGL}   \makeop{Corr}  \makeop{PSL}   \makeop{sgn}   \makeop{Spf}
\makeop{Spec}  \makeop{Tr}    \makeop{Nr}    \makeop{Fr}    \makeop{disc}
\makeop{Proj}  \makeop{supp}  \makeop{ker}   \makeop{im}    \makeop{dom}
\makeop{coker} \makeop{Stab}  \makeop{SO}    \makeop{SL}    \makeop{SL}
\makeop{Cl}    \makeop{cond}  \makeop{Br}    \makeop{inv}   \makeop{rank}
\makeop{id}    \makeop{Fil}   \makeop{Frac}  \makeop{GL}    \makeop{SU}
\makeop{Trd}   \makeop{Sp}    \makeop{Tr}    \makeop{Trd}   \makeop{diag}
\makeop{Res}   \makeop{ind}   \makeop{depth} \makeop{Tr}    \makeop{st}
\makeop{Ad}    \makeop{Int}   \makeop{tr}    \makeop{Sym}   \makeop{can}
\makeop{length}\makeop{SO}    \makeop{torsion} \makeop{GSp} \makeop{Ker}
\makeop{Adm}
\def\makebb#1{\expandafter\def
  \csname bb#1\endcsname{{\mathbb{#1}}}\ignorespaces}
\def\makebf#1{\expandafter\def\csname bf#1\endcsname{{\bf
      #1}}\ignorespaces} 
\def\makegr#1{\expandafter\def
  \csname gr#1\endcsname{{\mathfrak{#1}}}\ignorespaces}
\def\makescr#1{\expandafter\def
  \csname scr#1\endcsname{{\EuScript{#1}}}\ignorespaces}
\def\makecal#1{\expandafter\def\csname cal#1\endcsname{{\mathcal
      #1}}\ignorespaces} 

\def\doLetters#1{#1A #1B #1C #1D #1E #1F #1G #1H #1I #1J #1K #1L #1M
                 #1N #1O #1P #1Q #1R #1S #1T #1U #1V #1W #1X #1Y #1Z}
\def\doletters#1{#1a #1b #1c #1d #1e #1f #1g #1h #1i #1j #1k #1l #1m
                 #1n #1o #1p #1q #1r #1s #1t #1u #1v #1w #1x #1y #1z}
\doLetters\makebb   \doLetters\makecal  \doLetters\makebf
\doLetters\makescr 
\doletters\makebf   \doLetters\makegr   \doletters\makegr
     \def\qed{\qedmark\medbreak}%
\def\qedmark{{\enspace\vrule height 6pt width 5pt depth 1.5pt}}%

\normalsize

\makeop{Bl}

\def\Fpbar{\overline{\bbF}_p}
\def\Fp{{\bbF}_p}

\newcommand{\Z}{\mathbb Z}
\newcommand{\Q}{\mathbb Q}

\newcommand{\F}{\mathbb F}

\newcommand{\<}{\langle}   
\renewcommand{\>}{\rangle} 

\nc{\embed}{\hookrightarrow}

\newcommand{\ch}{characteristic }

\nc{\ol}{\overline}
\nc{\wt}{\widetilde}
\nc{\opp}{\mathrm{opp}}

\makeop{Ram}
\makeop{Rep}

\begin{document}
\renewcommand{\thefootnote}{\fnsymbol{footnote}}
\setcounter{footnote}{-1}
\numberwithin{equation}{section}


\title{On Hermitian forms over dyadic non-maximal local orders}
\author{Chia-Fu Yu}
\address{
Institute of Mathematics, Academia Sinica and NCTS (Taipei Office)\\
6th Floor, Astronomy Mathematics Building \\
No. 1, Roosevelt Rd. Sec. 4 \\ 
Taipei, Taiwan, 10617} 
\email{chiafu@math.sinica.edu.tw}


\begin{abstract}
We reduce a study of polarized abelian varieties over finite fields to
the classification problem of skew-Hermitian modules 
over (possibly non-maximal) local orders. 
The main result of this paper gives a complete classification of these
skew-Hermitian modules in the case when the ground ring is a
dyadic non-maximal local order.
\end{abstract} 

\maketitle


\section{Introduction}
\label{sec:01}

Let $p$ be a rational prime number. Let
$R:=\Z_2[X]/(X^2+p)=\Z_2[\pi]$, an order of the $\Q_2$-algebra 
$E:=\Q_2[X]/(X^2+p)$, 
where $\Z_2$ is the ring of $2$-adic integers, and $\pi$ is the image of
$X$ in $R$. Denote by $a \mapsto \bar
a $ the non-trivial involution on $E$ and $O_E$ the ring of integers
in $E$. By a skew-Hermitian module over $R$
we mean a $\Z_2$-free finite $R$-module $M$ together with a $\Z_2$-valued
non-degenerate alternating pairing
\[ \psi: M\times M\to \Z_2 \]
such that $\psi(ax,y)=\psi(x,\bar a y)$ for all $a\in R$ and $x,y\in
M$. If $M$ is self-dual with respect to the pairing $\psi$, then it is
called a self-dual skew-Hermitian module. In the paper we study the
classification of self-dual skew-Hermitian modules over $R$. As an
elementary fact, the ring $R$ is the maximal order if and only if $p=2$
or $p\equiv 1\ (\!\mod 4)$. It is easier to handle the case where $R$ is
maximal; the classification is known even for any non-Archimedean
local maximal order of \ch not 2, due to Jacobowitz
\cite{jacobowitz:hermitian}. We give an exposition of the
classification in Section~\ref{sec:02}, for the reader's convenience.
The main part of this paper treats the
the case where $R$ is not maximal, that is, 
the case $p\equiv 3 \ (\!\mod 4)$. We now describe the main results.

Let $r\ge 1$ be an integer. Let $S_r$ be the set of symmetric matrices
in $\GL_r(\F_2)$. Define the equivalence relation $\sim$ on $S_r$ by 
$A\sim B$, for $A, B\in S_r$, if there exists a matrix $P\in
\GL_r(\F_2)$ such that $B=P^t A P$. Denote by $S_r/\!\sim$ the set of
equivalence classes in $S_r$. 
Define the integers $m_r$ for $r\ge 0$ by 
$m_0:=1$ and
\begin{equation}
  \label{eq:11}
  m_r:=\# S_r/\!\sim, \quad \forall\, r\ge 1. 
\end{equation}

\begin{thm}\label{11}
Assume $p\equiv 7 \ (\!\mod 8)$.
  There are 
  \begin{equation}
    \label{eq:12}
    \sum_{r=0}^n m_r
  \end{equation}
non-isomorphic self-dual skew-Hermitian modules over $R$ of
$\Z_2$-rank $2n$. 
\end{thm}

%


\begin{thm}\label{12}
Assume $p\equiv 3 \ (\!\mod 8)$.
  There are 
  \begin{equation}
    \label{eq:14}
    \sum_{r=0}^n m_r
  \end{equation}
non-isomorphic self-dual skew-Hermitian modules over $R$ of
$\Z_2$-rank $2n$. 
\end{thm}       

The proofs are given in Sections~\ref{sec:05} and \ref{sec:06}. 
Though the statements of
Theorems~\ref{11} and \ref{12} look the same, the structures in the
classification are different. It is easy to compute the integers
$m_r$; see Lemma~\ref{57}. The classification problem for
($\Z_2$-valued) skew-Hermitian modules is equivalent to the same
problem for Hermitian
modules. Indeed, let $(M,\psi)$ be a skew-Hermitian module over
$R$. Then there is a unique Hermitian form
\[ \varphi:M\times M\to 2^m R \]
such that 
\[ \psi(x,y):=\Tr_{E/\Q_2} \pi \varphi(x,y),\quad  \forall\,x,y\in M. \] 
Here $m$ is the smallest integer such that $R^\vee\subset 2^m R$,
where $R^\vee$ 
is the dual lattice of $R$ for the pairing $(a,b)\mapsto
\Tr_{E/\Q_2}(ab)$; in fact $m=0$ of $-1$ when $R$ is maximal or not,
cf. \S~\ref{sec:61}. 
Conversely, given a Hermitian module $(M,\varphi)$ we get a
($\Z_2$-valued) skew-Hermitian module $(M,\psi)$ by setting
\[ \psi(x,y):=\Tr_{E/\Q_2} \pi \varphi(x,y),\quad \forall\, x,y\in
M. \]
  

It is worth noting that when $p\equiv 3 \ (\! \mod 4)$, the ground ring
$R$ is not hereditary (see Remark~\ref{44}), 
nor the condition that $a+\bar a=1$ for some
$a\in R$ is not fulfilled. Therefore, results in the paper complement
those of Riehm \cite{riehm:hermitian}, and 
a general Witt type cancellation theorem \cite[Theorem
3]{bayer-fluckiger-fainsilber} obtained by Bayer-Fluckiger and
Fainsilber; see Propositions~\ref{54} and \ref{67}. \\



The motivation of this work is to determine the Tate modules of
certain abelian varieties over finite fields as Galois modules. The
reader who is not familiar with abelian varieties may consult the
reference Mumford~\cite{mumford:av}. 
An abelian variety $A$ over a field $k$ of
\ch $p$ is said to be {\it superspecial} if it is isomorphic to a
product of supersingular elliptic curves over an algebraic closure
$\bar k$ of $k$. 
Let $\Sigma_n(\Fp)$ denote the set of isomorphism classes of $n$-dimensional
superspecial abelian varieties $(A,\lambda)$ over $\Fp$ together with 
a principal polarization $\lambda$ over $\Fp$ 
such that $\pi_A^2=-p$, where $\pi_A$ is the Frobenius
endomorphism of $A$.
Suppose
 $(A,\lambda)$ is an object in $\Sigma_n(\Fp)$.
For any prime $\ell\neq p$, the associated Tate
module $T_\ell(A)$ is a free $\Z_\ell$-module of rank
$2n$ together
with a continuous action $\rho_A$ of 
the Galois group $\calG:=\Gal(\Fpbar/\Fp)$
and a Galois equivariant self-dual alternating pairing (the Weil pairing)
\[ e_\ell:T_\ell(A)\times T_\ell(A)\to \Z_\ell(1), \]
where 
\[ \Z_\ell(1):=\lim_{\leftarrow} \mu_{\ell^m}(\Fpbar) \]
is the Tate twist. Let $\sigma:x\mapsto x^p$ be the Frobenius
automorphism of 
$\calG$. Since $\sigma x=\pi_A x$ for all $x\in T_\ell(A)$, the action
of the Galois group $\calG$ on the Tate module $T_\ell(A)$ factors through the
quotient
\[ \Z_\ell[\calG]\to  \Z_\ell[X]/(X^2+p)=\Z_\ell[\pi_A]. \]
 The pairing $e_\ell$ induces
an involution 
$a\mapsto \bar a$ on $\Z_\ell[\pi_A]$ by the adjoint. 
It follows from $\pi_A \ol \pi_A=p$ and $\pi_A^2=-p$ that this
involution is non-trivial. 
Fix an isomorphism $\Z_\ell(1)\simeq \Z_\ell$ as
$\Z_\ell$-modules.  The problem of classifying Tate modules of superspecial
abelian varieties in $\Sigma_n(\Fp)$ as $\calG$-modules together to
the Weil pairing amounts to
\begin{enumerate}
\item classifying self-dual skew-Hermitian modules
over the ring $\Z_\ell[X]/(X^2+p)$, and
\item determining the image of $\Sigma_n(\Fp)$ in the set 
of isomorphism classes of skew-Hermitian modules.
\end{enumerate}

The second problem seems to be hard.
Besides, in order to state the result of the second problem, 
one needs an explicit description of the classification.
In this paper we limit ourselves to the first problem.


As mentioned before, if the ring $\Z_\ell[X]/(X^2+p)$
is the maximal order in the (commutative and semi-simple)
$\Q_\ell$-algebra $\Q_\ell[X]/(X^2+p)$, then the classification is
known. As well-known the ring $\Z_\ell[X]/(X^2+p)$ is maximal except when 
$\ell=2$ and $p\equiv 3 \ (\!\mod 4)$, which is the most complicated
case. Theorems ~\ref{11} and \ref{12} give the answer to the first
probelm in this case. An immediate consequence of Theorems~\ref{11}
and \ref{12} (also using Lemma~\ref{57}) is the following.

\begin{thm}\label{13}
  Notation as above. Assume $p\equiv 3 \ (\!\mod 4)$. There are at
  most $n+\lfloor n/2 \rfloor$ isomorphism classes of polarized
  $2$-adic Tate modules $(T_2(A),\rho_A, e_2)$ for all objects
  $(A,\lambda)$ in $\Sigma_n(\Fp)$.
\end{thm}

The paper is organized as follows. In Section~\ref{sec:02} we give an
exposition of the classification of Hermitian forms over
non-Archimedean local maximal orders of \ch not 2, following
Jacobowitz \cite{jacobowitz:hermitian}. Section 3 gives the complete
classification of $R$-modules. Sections 4 and 5 treat the
classification of skew-Hermitian modules in the cases
$p\equiv 7 \ (\!\mod 8)$ and $p\equiv 3 \ (\!\mod 8)$ separately.    


\begin{thank}
  The manuscript is prepared during the
  author's stay at l'Institut des Hautes \'Etudes Scientifiques. 
  He acknowledges the institution for kind hospitality and excellent working
  conditions. The research was partially supported by grants NSC
  97-2115-M-001-015-MY3 and AS-99-CDA-M01.
\end{thank}

\section{Hermitian forms over local fields}
\label{sec:02}
In this section we give an exposition of the classification of 
Hermitian forms over local
maximal orders, for the reader's convenience. 
Our reference is Jacobowitz \cite{jacobowitz:hermitian}.

\subsection{Hermitian spaces}
\label{sec:21}

Let $F$ be a non-Archimedean local field, that is, it is complete with
respect to a discrete valuation ring with finite residue field, which is
assumed of \ch not 2. Let $O_F$ be the ring of integers of $F$, and
$\pi_F$ a uniformizer of $O_F$. Let $E$ be a quadratic field
extension of $F$. Let $O_E$ be the ring of integers in $E$, and $\pi$
a uniformizer of $O_E$. Write $a\mapsto \bar a$ for the non-trivial
automorphism of $E$ over $F$, and let $\pi=\pi_F$ if the extension
$E/F$ is unramified. Let $v$ be the normalized valuation on $E$
such that $v(\pi)=1$. 

By a Hermitian space (over $E$) we mean a finite-dimensional vector $V$
over $E$, together with a bi-additive pairing 
\[ h:V\times V\to E \]
such that $h(a x,by)=a\bar b h(x,y)$ for all $a, b\in E$ and $x,y\in
V$. The pairing $h$ is called a Hermitian form. It is called
non-degenerate if the induced linear map $V\to V^*$ by $x\mapsto
h(\cdot, x)$ is injective (and hence isomorphic). By a Hermitian module
(over $O_E$) we mean a finite free $O_E$-module $L$ together with a
hermitian form $h$, ususally assumed $h(L,L)\subset O_E$. 
It is called non-degenerate if the same property holds
for $L$, and called {\it unimodular} if the induced 
map $L\to L^*$ is also surjective. 
A full rank submodule in a Hermitian space is usually called a lattice.
A decomposition of a Hermitian module (resp. space) $L$ into
submodules (resp. subspaces) $L_1$ and $L_2$:
\[ L=L_1\oplus L_2 \]
means that $L=L_1+L_2$, $L_1\cap L_2=0$ and $h(x,y)=0$ for all $x\in
L_1$ and $y\in L_2$. For a submodule (resp. subspace) $L_1$ of $L$,
denote by $L_1^{\bot}$ the orthogonal complement of $L_1$ in $L$. 

\begin{lemma}
  Any Hermitian space $(V,h)$ has a decomposition 
\[ V=V^{\bot} \oplus V_1, \]
where $V_1$ is any subspace complemented to the null subspace $V^{\bot}$.
The Hermitian subspace $V_1$, if non-zero, is
non-degenerate and the projection $V_1 \to V/V^{\bot}$ is isometric. 
\end{lemma}

From now on, we assume that $(V,h)$ is non-degenerate. 
The determinant or discriminant of $V$ (resp. of a lattice $L$),
denoted by $d V$ 
(resp. $d L$), is defined as
$\det(h(x_i,x_j))$, where $x_1,\dots, x_n$ is an $E$-basis for $V$
(resp. an $O_E$-basis for $L$). The determinant $dV$ (resp. $dL$)
is unique up to an element in $N_{E/F}(E^\times)$
(resp. $N_{E/F}(O_E^\times)$). Write $dV_1\simeq dV_2$
(resp. $dL_1\simeq dL_2$) if there is an element $a\in
N_{E/F}(E^\times)$ (resp. $a\in N_{E/F}(O_E^\times)$) such that
$dV_2=a\cdot dV_1$ (resp. $dL_2=a\cdot dL_1$). 
  
\begin{thm}\
\begin{enumerate}
\item   We have $V=Ex_1\oplus\dots \oplus Ex_{n-1} \oplus E x_n$ with
    $h(x_i,x_i)=1$ for $i=1,\dots, n-1$ and $h(x_n, x_n)=dV$.
\item Two Hermitian spaces $V_1$ and $V_2$ are isometric if and only
  if $\dim V_1=\dim V_2$ and $dV_1\simeq dV_2$.   
\end{enumerate}

\end{thm}
\begin{proof}
  This is Theorem~3.1 of \cite{jacobowitz:hermitian}. The statement 
  (1) follows
  from the fact that any non-degenerate quadratic form of rank $\ge 4$
  over $F$ represents every non-zero element in $F$. The statement (2)
  follows from (1). 
\end{proof}

\subsection{General properties of Hermitian lattices}
\label{sec:22}

Let $(L,h)$ be a Hermitian lattice. For any elements $x_1,\dots,
x_n$ in $L$, denote by $<x_1,\dots, x_n>_{O_E}$ the $O_E$-submodule
generated by $x_1,\dots, x_n$.  
If $L$ has an orthogonal basis
$x_1, \dots, x_n$ with $h(x_i,x_i)=\alpha_i$, we may write
\[ L =< x_1>\oplus \dots \oplus <x_n> \simeq (\alpha_1)\oplus \dots
\oplus (\alpha_n). \]

\begin{defn}\
\begin{enumerate}
\item A vector $x\in L$ is called {\it maximal} if $x\not\in \pi L$. 
\item Let $sL:=\{h(x,y)|x,y\in L\}$, and let $nL$ be the $O_E$-ideal
  generated by all elements $h(x,x)$ with $x\in L$. Clearly, one has $nL\subset
  sL$. 
\item We call $L$ {\it normal} if $nL=sL$, and {\it subnormal}
  otherwise. 
\item A lattice $L$ is called {\it $\pi^i$-modular}, where $i\in \Z$, if
  $h(x,L)=(\pi^i)$ for every maximal vector $x\in L$; $L$ is called
  {\it modular} if it is $\pi^i$-modular for some $i\in \Z$.   
\end{enumerate}
\end{defn}
 
Clearly, if $L_1$ and $L_2$ are $\pi^i$-modular, then so as the
  direct sum $L_1\oplus L_2$. Any rank one lattice is modular. The
  lattice $L=<x,y>$ with $v(h(x,y))=i, v(h(x,x)>i$, and $v(h(y,y))>i$
  is a $\pi^i$-modular plane. We write 
\[ 
\begin{pmatrix}
  a & b \\ \bar b & c
\end{pmatrix} \]
for the Hermitian plane $(O_E^2, h)$ with $h(e_1,e_1)=a$,
$h(e_1,e_2)=b$ and $h(e_2,e_2)=c$.  
For any $i\in \Z$, define the
  hyperbolic plane $H(i)$ to be 
\[ 
\begin{pmatrix}
  0 & \pi^i \\ \bar \pi^i & 0
\end{pmatrix}. \]

 
\begin{prop}\label{24}
  There is an $O_E$-basis $x_1,\dots, x_{r}, y_1,\dots, y_s,
  z_1,\dots, z_s$ such that 
\[ V=<x_1>\oplus \dots \oplus <x_r>\oplus <y_1,z_1>\oplus \dots \oplus
  <y_s,z_s> \]
and all components are modular.  
\end{prop}
\begin{proof}
  This is Proposition 4.3 of \cite{jacobowitz:hermitian}.
\end{proof}

\begin{prop}\label{25}
  A $\pi^i$-modular lattice $L$ has an orthogonal basis if any of
  the following conditions holds:
\begin{enumerate}
\item $L$ has odd rank.
\item $L$ is normal.
\item $i=0$ and there is an element $a\in E$ such that $v(a)=v(a+\bar
  a)=0$. 
\end{enumerate}
\end{prop}
\begin{proof}
  This is Proposition 4.4 of \cite{jacobowitz:hermitian}.
\end{proof}

\begin{defn}\
\begin{enumerate}
\item For any integer $j\in \Z$, define 
\[ L_{(j)}:=\{x\in L\, |\, h(x,L)\subset (\pi^j)\, \}.\]
\item A decomposition $L=\oplus_{1\le \lambda\le  t} L_\lambda$ of $L$ is
  called a {\it Jordan splitting} if each $L_\lambda$ is modular and
  \[ sL_1\supsetneq \dots, \supsetneq sL_t.\] 
  Two Jordan
  splittings 
\[ L=\oplus_{1\le \lambda\le t} L_\lambda,\quad K=\oplus_{1\le
  \lambda\le T} K_\lambda \]
are said to be of the {\it same type} if $t=T$, and for each
$\lambda$, one has
\[ sL_\lambda=sK_\lambda,\quad \rank L_\lambda=\rank K_\lambda, \]
and $L_\lambda$ and $K_\lambda$ are both normal or both subnormal. 
\end{enumerate}
   
\end{defn}

It follows from Proposition~\ref{24} that every lattice has a Jordan
splitting. Any two Jordan splittings of a lattice $L$ are of the same
type (see \cite[p.~449]{jacobowitz:hermitian}). For a Jordan splitting
$L=\oplus_\lambda L_\lambda$, define the integers $s(\lambda)$ and
$u(\lambda)=u_L(\lambda)$ for $\lambda=1,\dots, t $ by
\[ sL_\lambda=(\pi^{s(\lambda)}), \quad nL_{(s(\lambda))}
  =(\pi^{u(\lambda)}), \]
and the $O_E$-ideals $\grf(\lambda)$ for $\lambda=1,\dots,t-1$ by 
\[ \grf(\lambda):=(\pi^{u(\lambda)+u(\lambda+1)-2s(\lambda)}). \]
Since any two splittings are of the same type, the invariant
$\{s(\lambda)\}$, and hence the invariants $\{u(\lambda)\}$ and
$\{\grf(\lambda)\}$, are independent of the choice of Jordan
splittings.

\begin{thm} Suppose $E/F$ is unramified. Then 

\begin{enumerate}\label{27}
\item There is an element $a$ in $E$ such that $v(a)=v(a+\bar a)=0$. 
\item Any $\pi^i$-modular lattice $L$ is isomorphic to $(\pi^i)\oplus
  \dots, \oplus (\pi^i).$
\item Any two lattices are isometric if and only if they are of the
  same type.   
\end{enumerate}
\end{thm}
\begin{proof}
  This is Theorem~7.1 of  \cite{jacobowitz:hermitian}.  
\end{proof}

\begin{thm}\label{28} Suppose $E/F$ is ramified and non-dyadic.

\begin{enumerate} 
\item Let $L$ be a $\pi^i$-modular lattice of rank $n$. \\
If $i=2d$ is even, then
\[ L\simeq (\pi^d)\oplus \dots \oplus (\pi^d)\oplus (\pi^{-(n-1)d}
dL). \]
If $i$ is odd, then
\[ L\simeq H(i)\oplus \dots \oplus H(i). \]

\item Let $L$ and $K$ be two lattices with Jordan splittings $\oplus
  L_\lambda$ and $\oplus K_\lambda$, respectively. Then $L\simeq K$ if
  and only if $L$ and $K$ are of the same type and $dL_j\simeq dK_j$
  for every index $j$ for which $s(j)$ is even.  
\end{enumerate}
  
\end{thm}
\begin{proof}
  These are Proposition 8.1 and Theorem 8.2 of
  \cite{jacobowitz:hermitian}. \\
\end{proof}

Theorems~\ref{27} and \ref{28} complete the classification in the
cases where $E/F$ is unramified, or $E/F$ is ramified and
non-dyadic. It remains to describe the ramified dyadic case.

\subsection{The ramified dyadic case: modular lattices.}
\label{sec:23}

In the remaining of this section we assume that $E/F$ is ramified and dyadic. 

\begin{prop}\label{29} 
 Let $L$ is a $\pi^i$-modular lattice of rank $\ge 3$. Then 
\[ L\simeq L_0 \oplus H(i)\oplus \dots \oplus H(i). \]
where $L_0$ is of rank one or two with $nL_0=nL$. 
\end{prop}
\begin{proof}
  This is Proposition 10.3 of  \cite{jacobowitz:hermitian}.  
\end{proof}

\begin{prop}\label{210}
  Let $L_1$ and $L_2$ be two $\pi^i$-modular lattices. If $L_1\oplus
  H(i)\simeq L_2\oplus H(i)$, then $L_1\simeq L_2$. 
\end{prop}
\begin{proof}
  This is Proposition 9.3 of  \cite{jacobowitz:hermitian}.\\ 
\end{proof}

By Propositions~\ref{29} and \ref{210}, the classification of modular
lattices is reduced to the case of planes and $i=0$ or $1$.\\

We need a classification of dyadic ramified extensions. 
For an element $a$ 
in $O_F^{\times}$, denote by $\grd_{F}(a)$ the smallest
$O_F$-ideal $J$ such that $a \mod J$ is a square. If $a$ is a square,
then $\grd_{F}(a)=0$. Since the squaring $x\mapsto x^2$ is an automorphism on
$O_F/(\pi_F)$, one has $\grd_F(a)\subset (\pi_F)$ for any $a$. 
It is known that the $O_F$-ideals
occurring as $\grd_F(a)$ are precisely $0, (4)$, and all
$(\pi_F^{2k+1})$ with $0< 2k+1 < v_F(4)$ where
$v_F(\pi_F)=1$. Furthermore, one has 
\[ \grd_F(1+\pi_F^{2k+1} \delta)=(\pi_F^{2k+1}) \] 
for such integers $k$, where $\delta$ is a unit. 

Write $E=F(\sqrt{\theta})$, where $\theta$ is a non-square unit or 
prime element. We have (\cite[p.~451]{jacobowitz:hermitian})
\begin{itemize}
\item [(a)] $\theta$ is a prime element, or
\item [(b)] $\theta$ is a unit and $\grd_F(\theta)=(\pi_F^{2k+1})$
  with $0< 2k+1< v_F(4)$. In this case,
  \[ E=F((1+\pi_F^{2k+1}\delta)^{1/2})\]
 for some unit $\delta\in O_F$. 
\end{itemize}
We refer the case (a) as {\it ramified prime}, (``R-P''), and the case
(b) as {\it ramified unit}, (``R-U''). 

We use a notation from \cite[p.~450]{jacobowitz:hermitian}: to
indicate an unspecified element $a$ in $E$ with $v(a)\ge v(b)$, we
shall write $a=\{b\}$; to indicate an unspecified element $a$ with
$v(a)=v(b)$, write $a=[b]$. 

\begin{prop}\label{211}\
\begin{enumerate}
\item If $L$ is $\pi^i$-modular, then $nL\supset nH(i)$.
\item If $a\in F$ is any element in $nH(i)$, then 
\[ 
\begin{pmatrix}
  0 & \pi^i \\ \bar \pi^i & a 
\end{pmatrix}\simeq H(i). \]
\end{enumerate}
\end{prop}
\begin{proof}
  This is Proposition 9.1 of  \cite{jacobowitz:hermitian}.
\end{proof}
\begin{prop}\label{212} Let $L$ be a $\pi^i$-modular plane, $i=0$ or
  $1$, with $nL=nH(i)$.
\begin{enumerate}
\item If $L$ is isotropic, then $L\simeq H(i)$.
\item In R-P with $i=1$ and in R-U with $i=0$, $L$ must be
  isotropic, particularly $L\simeq H(i)$. In the other two cases, if
  $L$ is anisotropic, then $(h(x,x))=nL$ for any maximal vector $x\in
  L$.
\item If $K$ is another $\pi^i$-modular plane, with $nK=nL$ and
  $dK\simeq dL$, then $K\simeq L$. 
\end{enumerate}
\end{prop}
\begin{proof}
  This is Proposition 9.2 of  \cite{jacobowitz:hermitian}.
\end{proof}\

Propositions~\ref{211} and \ref{212} handle the case where $L$ is a
$\pi^i$-modular plane with $nL=nH(i)$, the minimal case. The other
case $nL\neq nH(i)$ is treated in the following proposition.

\begin{prop}\label{213} Let $L$ is a $\pi^i$-modular plane, $i=0$, or
  $1$, with $nL=(\pi^{2m})\supsetneq nH(i)$.
\begin{enumerate}
\item If $L$ is normal, then $L\simeq (1)\oplus (dL)\simeq 
  \begin{pmatrix}
    1 & 1 \\ 1 & \{1\}
  \end{pmatrix}.$
\item If $L$ is subnormal, then 
\[ L\simeq 
\begin{pmatrix}
  \pi_F^{m} & \pi^i \\ \bar \pi^i & \{a\}
\end{pmatrix}, \]
where $a=4 \pi_F^{-m+i}$ in R-P, and $a=4\pi_F^{-2k-m+i-1}$ in R-U.
\end{enumerate}
\end{prop}
\begin{proof}
  This is Proposition 10.2 of  \cite{jacobowitz:hermitian}.
\end{proof}\

Using Propositions~\ref{29}--\ref{213}, one can conclude the following
characterization for modular lattices (see
\cite[Proposition 10.4]{jacobowitz:hermitian}).

\begin{thm}\label{214}
  Let $L$ and $K$ be $\pi^i$-modular lattices, $i\in \Z$. Then
  $L\simeq K$ if and only if $\rank L=\rank K$, $nL=nK$, and $dL=dK$. 
\end{thm}

\subsection{The ramified dyadic case: the invariants.}

Let $L$ and $K$ be two Hermitian lattices, and let $I\subset O_E$ be a
proper ideal. We write $dL/dK \simeq 1\ (\! \mod I)$ if $v(dL)=v(dK)$
and there are bases $x_i$ and $y_i$ for $L$ and $K$, respectively,
such that 
\[ \det(h_L(x_i,x_j))/\det(h_K(y_i,y_j)) \equiv 1\ (\! \mod I). \]

The following (\cite[Theorem 11.4]{jacobowitz:hermitian}) 
gives a complete classification of Hermitian forms over local maximal
orders by the invariants.

\begin{thm}\label{215}
  Let $L$ and $K$ be two Hermitian lattices. Suppose 
\[ L=\oplus_{1\le \lambda\le t} L_\lambda \quad \text{and} \quad 
   K=\oplus_{1\le \lambda\le T} K_\lambda \]
are any Jordan splittings. Then $L$ and $K$ are isometric if and only
if the following four conditions hold:
\begin{enumerate}
\item $L$ and $K$ are of the same type.
\item $dL\simeq dK$.
\item $u_L(\lambda)=u_K(\lambda)$ for all $\lambda=1,\dots, t$.
\item For all $j=1,\dots, t-1$, one has
\[ d(L_1\oplus \dots \oplus L_j)/d(K_1\oplus \dots \oplus K_j)\simeq 1
\quad (\!\!\mod \grf(\lambda)\,). \]
\end{enumerate}
\end{thm}
 


\section{Unpolarized cases}
\label{sec:04}
In this section we assume that $p\equiv 3 \ (\,{\rm mod}\, 4)$. 
Recall that $R=\Z_2[X]/(X^2+p)=\Z_2[\pi]$ and $E=R\otimes_{\Z_2}
\Q_2=\Q_2[\pi]$. Let $O_E$ 
be the ring of integers of $E$. We have 
\[ O_E=\Z_2[\alpha]=\Z_2[X]/(X^2+X+(p+1)/4). \]
Put $\omega:=\pi-1$, and one has 
\[ R=\Z_2[\omega]=\Z_2[X]/(X^2+2X+(p+1)\,)\quad \text{and}
\quad 2\alpha=\omega.  \] 
We shall classify $R$-modules $M$ which is finite and free as
$\Z_2$-modules. Write $<x_1,\dots, x_m>_R$ for the $R$-submodule of
$M$ generated by elements $x_1, \dots, x_m$. 

We divide the classification into two cases:\\

{\bf Case (a):} $p\equiv 3 \ (\,{\rm mod}\, 8)$. In this case, the
algebra $E$ is
a unramified 
quadratic extension of $\Q_2$. We have (at least) 
two indecomposable $\Z_2$-free finite $R$-modules: 
$R$ and $O_E$ as $R$-modules. The $R$-module structure of $O_E$ is
given as follows: write $O_E=<1,\alpha>_{\Z_2}$, then 
\begin{equation}
  \label{eq:41}
  \omega 1= 2\alpha\quad\text{and} \quad 
\omega \alpha= -2\alpha - (p+1)/2. 
\end{equation}
If $M=R^{\oplus r}\oplus O_E^{\oplus s}$, then the non-negative integers $r$ and $s$ are
uniquely determined by $M$. 
Indeed, we have $r+s=\dim_{E} M\otimes_{\Z_2} \Q_2$, 
and $M/(2,\omega)M=(\F_2)^r\oplus (\F_2\oplus \F_2)^s. $  \\

{\bf Case (b):} $p\equiv 7 \ (\,{\rm mod}\, 8)$. In this case,
the algebra $E$ is equal to $\Q_2\times \Q_2$. Write 
\[ X^2+X+(p+1)/4=(X-\alpha_1)(X-\alpha_2), \]
where
$\alpha_1, \alpha_2\in \Z_2$. By switching the order, we may assume
that $\alpha_1$ is a unit and $\alpha_2\in 2\Z_2$. We have the isomorphisms
\[ O_E=\Z_2[\alpha] \simeq O_E/(\alpha-\alpha_1)\times
O_E/(\alpha-\alpha_2)\simeq \Z_2\times \Z_2. \] 
Therefore, 
\[ X^2+2X+(p+1)=(X-2\alpha_1)(X-2\alpha_2). \]
We have (at least) three indecomposable $\Z_2$-free finite
$R$-modules: 
\[ R, \quad R/(\omega-2\alpha_1),\quad\text{and}\quad R/(\omega-2\alpha_2).\] 
Among them, we have 
\[ O_E\simeq R/(\omega-2\alpha_1)\oplus R/(\omega-2\alpha_2) \]
as $R$-modules. 
If $M=R^r\oplus [R/(\omega-2\alpha_1)]^s\oplus
[R/(\omega-2\alpha_2)]^t$, then the non-negative integers $r$, $s$ and
$t$ are uniquely determined by $M$. Indeed, we have 
\[ \rank_{\Z_2} M=2r+s+t, \quad 
M/(2,\omega)M=\F_2^r\oplus \F_2^s\oplus \F_2^t, \]
and  
\[ M/(\omega-2\alpha_1)M=[R/(\omega-2\alpha_1)]^{r+s}\oplus (\F_2)^t. \]\


Conversely, we show that the indecomposable
finite $R$-modules described in
Cases (a) and (b) exhaust all possibilities.
  
\begin{thm}\label{41} 
Let $M$ be a $\Z_2$-free finite $R$-module. Then \
\begin{enumerate}
  \item Assume $p\equiv 3 \ (\,{\rm mod}\, 8)$ {\rm (Case
      (a))}. The $R$-module $M$ is isomorphic to 
    $R^r\oplus O_E^s$ 
    for some non-negative integers $r$ and $s$. Moreover, the integers 
    $r$ and $s$ are uniquely determined by $M$. 
  \item Assume $p\equiv 7 \ (\,{\rm mod}\, 8)$ {\rm (Case
      (b))}. The $R$-module $M$ is isomorphic to  
\[ R^r\oplus \left [ R/(\omega-2\alpha_1)\right ]^s\oplus
\left [ R/(\omega-2\alpha_2)\right ]^t\]
    for some non-negative integers $r$, $s$ and $t$. Moreover, the
    integers $r$, $s$ and $t$ are uniquely determined by $M$. 
\end{enumerate}
\end{thm}
\begin{proof} The unique determination of integers $r$, $s$ and $t$
    has been shown. We prove the first part of each statement. 

  (1) Let 
  \begin{equation}
    \label{eq:42}
    \ol M:=M/2M=(\F_2[\omega]/\omega^2)^r\oplus (\F_2)^{2s}
  \end{equation}
be the decomposition as $R/2R=\F_2[\omega]/(\omega^2)$-modules. We
first show that if $s=0$, then $M\simeq R^r$. Since $r=\dim
M\otimes_{R} \F_2=\dim_E M\otimes_{R} E$ and $R$ is a local Noetherian
domain, the module $M$ is free. 

Now suppose $s>0$. Choose an element $a \neq 0 \in (\F_2)^{2s}$
and let 
$x\in M$ be an element such that $\bar x=a$. As $\ol{\omega
  x}=0$, the element $\omega x/2\in M$. Put $M_1:=<x,\omega
x/2>_{\Z_2}$; it is an $R$-module and is isomorphic to $O_E$. Let
$\omega'$ be the conjugate of $\omega$; one has $\omega'=-2-\omega$
and $\omega \omega'=(1+p)$. Note that $(1+p)/4$ is a unit. Since $\bar
x\not\in \omega' \ol M$, one has $x\not\in \omega' M$. We show that
$\ol{\omega x/2}\neq 0$. Suppose not, then $\omega x=4y$ for some
$y\in M$. Applying $\omega'$, we get $x=\omega' y'$ for some $y'\in
M$, contradiction. Since $x$ and ${\omega x/2}$ are $\Z_2$-linearly 
independent and are not divisible by $2$, 
the $\F_2$-vector space $\ol M_2=<\bar x, \ol{\omega
  x/2}>$ has 
dimension $2$, and hence the quotient $\ol M/\ol M_1$ has dimension
deceased by $2$. On the other hand, the $\Z_2$-rank of $M/M_1$ also
decreases by $2$. This shows that $M/M_1$ is free as $\Z_2$-modules. 
If the integer $s$ in (\ref{eq:42}) for $M/M_1$ is positive, then we
can find an $R$-submodule $M_2=<x_2, \omega x_2/2>_{\Z_2}\simeq O_E$
not contained in the vector space $E M_1$ 
such that $M/(M_1+M_2)$ is free as $\Z_2$-modules. 
Continuing this process, we get $R$-submodules $M_1,\dots, M_{s'}$,
which are isomorphic to $O_E$, such that $M_1+\dots+M_{s'}=M_1\oplus
\dots\oplus M_{s'}$ and $M/(M_1+\dots + M_{s'})$ is a free
$R$-module. It follows that $M\simeq O_E^{s'}\oplus R^{r'}$. Since
$s'$ and $r'$ are uniquely determined by $M$ as before, the integers
$s'$ and $r'$ are actually equal to $s$ and $r$ in (\ref{eq:42}), 
respectively. This proves (1). 

(2) Let 
\[ M_1:=\{x\in M\, |\, (\omega-2\alpha_1)x=0\, \},\]
and 
\[ M_2:=\{x\in M\, |\, (\omega-2\alpha_2)x=0\, \}.\]
Using the relation
\[
2=(\omega-2\alpha_1)(2\alpha_2+1)^{-1}-
(\omega-2\alpha_2)(2\alpha_2+1)^{-1}, \] 
one shows that $2M\subset M_1+M_2$, and hence the quotient
$M/(M_1+M_2)$ is an $\F_2$-vector space, say of dimension $r$.   
Let $x_1, \dots, x_r$ be elements of $M$ such that the images $\bar
x_1, \dots, \bar x_r$ form an $\F_2$-basis for $M/(M_1+M_2)$. Put
$F_0:=<x_1, \dots, x_r>_R$, which is isomorphic to $R^r$, as $\bar
x_i's$ form a basis for $F_0/(M_1+M_2)=F_0/(2,\omega)F_0\simeq
\F_2^r$. Now $(\omega-2\alpha_2)F_0\subset M_1$, we choose elements
$y_1, \dots, y_s$ in $M_1$ so that the images $\bar y_1, \dots, \bar
y_s$ form an $R/(\omega-2\alpha_1)$-basis for
$M_1/(\omega-2\alpha_2)F_0$, and put $F_1=<y_1,\dots, y_s>_R$. We have 
\[ M_1=(\omega-2\alpha_2)F_0\oplus F_1,\quad \text{and} \quad F_0\cap
F_1=0. \]
 Similarly, we have a free $R/(\omega-2\alpha_2)$-submodule 
$F_2$ of $M_2$, of rank $t$, such that
\[ M_2=(\omega-2\alpha_1)F_0\oplus F_2,\quad \text{and} \quad F_0\cap
F_2=0. \]  
We have $(F_0+F_1)\cap F_2=F_0\cap F_2=0$ and $M=F_0+F_1+F_2$, and
hence $M=F_0\oplus F_1\oplus F_2$. This proves (2). \qed  
\end{proof}


\begin{cor}\label{42}
  Assume $p\equiv 3 \ (\,{\rm mod}\, 4)$.
  Let $A$ be an 
  $n$-dimensional superspecial abelian variety $A$ over $\Fp$ with
  $\pi_A^2=-p$. 
  Then the Tate module
  $T_2(A)$ of $A$ is isomorphic to $R^r\oplus O_{E}^s$ for
  some non-negative integers $r$ and $s$ such that $r+s=n$. Moreover,
  the integers $r$ and $s$ are uniquely determined by $T_2(A)$. 
\end{cor}

\begin{proof}
  Note that the Tate space $V_2(A)=T_2(A)\otimes_{\Z_2} \Q_2$ is a free
  $E$-module of rank $n$; this follows from the fact that 
  $\tr(a; V_2(A))= n \tr (a; E)$ for all $a\in R$.
  It follows that the numbers $s$ and $t$ in Theorem~\ref{41} (2) 
  above are the same. Therefore, the corollary follows. \qed 
\end{proof}

\begin{lemma}\label{43}
  Assume $p\equiv 3 \ (\,{\rm mod}\, 4)$. Let $n\ge 1$ be an integer.
  For any non-negative integers $r$ and $s$ with $r+s=n$, there
  exists an $n$-dimensional superspecial abelian variety $A_r$ over
  $\Fp$ with $\pi_A^2=-p$ such that the Tate
  module $T_2(A_r)$ of $A_r$ is isomorphic to $R^r\oplus O_{E}^s$. 
\end{lemma}
\begin{proof}
  Choose a supersingular elliptic curve $E_0$ over $\F_p$ such that
  the endomorphism ring $\End_{\Fp}(E_0)$ is equal to the ring
  $O_{\Q(\sqrt{-p})}$ of integers in the imaginary quadratic field
  $\Q(\sqrt{-p})$, and a 
  supersingular elliptic curve $E_1$ over $\F_p$ such that 
  the endomorphism ring $\End_{\Fp}(E_0)$ is equal to 
   $\Z[\sqrt{-p}]$ 
  (see Waterhouse \cite[Theorem 4.2 (3), p.~539]{waterhouse:thesis}). 
  Put $A_r=E_1^r\times E_0^s$, then the superspecial
  abelian variety $A_r$ has the desired property. \qed 
\end{proof}

\begin{remark}\label{44} 
  We recall
  that an order $\Lambda$ (in a semi-simple separable algebra over a
  global or a local field) is called {\it left hereditary} if all its
  left ideals are projective as $\Lambda$-modules. One defines in a
  similar manner the notion of {\it right hereditary}, but it is known
  that these two notions are equivalent \cite[Theorem
  40.1]{reiner:mo}; we shall write {\it hereditary}. 

  The structure
  of hereditary orders is known (see \cite[Theorem 39.14]{reiner:mo})
  and they include in particular the maximal orders \cite[Section
  26]{curtis-reiner:1}. 

  It is worth noting that the ring $R$ is not hereditary. Indeed, the
  ideal $2O_E\simeq O_E$ is not a projective $R$-module, as the
  surjective map
\[ \pi:R\oplus R\to O_E, \quad (1,0)\mapsto 1,\ (0,1)\mapsto \alpha \]
  does not admit a section.  
\end{remark}

\section{Proof of Theorem~\ref{11}.}
\label{sec:05}

\subsection{}\label{sec:51}
In this section we assume that $p\equiv 7 \ (\,{\rm mod}\, 8)$. Thus,
the algebra $E$ is isomorphic to $\Q_2\times\Q_2$. Write
$\sigma_i:E\to \Q_2$ for the $i$th projection for $i=1,2$. Let
$(M,\psi)$ be a self-dual skew-Hermitian module over $R$ of
$\Z_2$-rank $2n$ and let $V:=M\otimes_{\Z_2} \Q_2$. For $i=1,2$, let
\begin{equation}
  \label{eq:51}
  V^i:=\{x\in V;\, ax=\sigma_i(a)x,\ \forall a\in E\,
  \}\quad\text{and}\quad  M^i:=V^i \cap M
\end{equation}
be the $\sigma_i$-components of $V$ and $M$, respectively.
It follows from the property $\psi(ax,y)=\psi(x,\bar a y)$ that each
$V^i$ is an isotropic subspace over $\Q_2$. Thus, $\dim V^i\le n$ for
$i=1,2$ and hence $\dim V^1=\dim V^2$. Therefore, $V$ is a free
$E$-module of rank $n$. By Theorem~\ref{41}, there are unique non-negative
integers $r$ and $s$ with $r+s=n$ such that 
\begin{equation}
  \label{eq:52}
  M\simeq R^{\oplus r} \oplus O_E^{\oplus s}
\end{equation}
as $R$-modules. Note that $r=0$ if and only if $M=M^1+M^2$, and we have
$M/(M^1+M^2)\simeq (\Z/2\Z)^r$. 

Define a self-dual skew-Hermitian module $(L, \varphi)$ as
follows. The $R$-module $L$ is 
$O_E=\Z_2\oplus \Z_2$. Put $e_1=(1,0)$ and
$e_2=(0,1)$, and set $\varphi(e_1,e_2)=1$.

\subsection{}\label{sec:52}
Suppose there exist elements $x\in M^1$ and $y\in M^2$ such that
$\psi(x,y)=1$. Then the submodule $M_1=<x,y>_R$ generated by $x$ and
$y$ is isomorphic to $L$ as skew-Hermitian modules and we have the
decomposition 
\begin{equation}
  \label{eq:53}
  M=M_1\oplus M_1^{\bot}
\end{equation}
as skew-Hermitian modules, where $M_1^{\bot}$ is the orthogonal
complement of $M_1$. Therefore, if $s=n$, then $M\simeq L^{\oplus n}$
as skew-Hermitian modules. 


\subsection{}\label{sec:53}
Recall that the polarization type of a non-degenerate alternating
pairing $\psi'$ on a free module $M'$ over a PID $R'$ of rank $r$ is
a tuple $(d_1,\dots, d_r)$ of elements in $R'$ with $d_1|\dots|d_r$
such that there exists a Lagrangian $R'$-basis $x_i, y_i$ for
$i=1,\dots, r$ 
such that $\psi'(x_i,y_i)=d_i$. The elements $d_i$ are unique up to a
unit in $R'$. The choice of the Lagrangian basis $\{x_i, y_i\}$ gives
rise to a splitting of $M'=M'_1\oplus M'_2$ into isotropic
submodules $M'_1:=<x_1, \dots, x_r>_{R'}$ and $M'_2:=<y_1, \dots,
y_r>_{R'}$.
Conversely, if $M'$ splits into 
the direct sum of two isotropic submodules $M'_1$ and $M'_2$. 
then a Lagrangian basis $\{x_i,y_i\}$ can be chosen so that $x_i\in
M'_1$ and $y_i\in M'_2$ for all $i=1,\dots, r$.


\begin{lemma}\label{51} \
\begin{enumerate}
\item   The polarization type of the pairing $\psi$ on the submodule
  $M^1+M^2$ as a $\Z_2$-module is $(1,\dots, 1,2,\dots, 2)$ 
  with multiplicity $s$ and $r$ for $1$ and $2$, respectively. 
\item There is a decomposition as skew-Hermitian modules
  \begin{equation}
    \label{eq:54}
   M\simeq M_1\oplus L^{\oplus s},  
  \end{equation}
where $M_1$ is a free $R$-submodule of rank $r$ which is self-dual
with respect to the pairing $\psi$.   
\end{enumerate}
\end{lemma}

\begin{proof}
The polarization type of the pairing $\psi$ on the submodule
  $M^1+M^2$ as a $\Z_2$-module has the form $(2^{a_1},\dots, 2^{a_n})$
  with integers $0\le a_1\le \dots \le a_n$. Since
  $|M/(M^1+M^2)|=2^r$, one has $\sum_{i=1}^n a_i=r$. If $a_1>0$, then
  $a_i=1$ for all $i$ and $r=n$. Below we show that $s=0$ implies
  $a_1>0$. This proves the case where $s=0$. 

  If $a_1=0$, then there exist elements $x\in
  M^1$ and $y\in M^2$ such that $\psi(x,y)=1$. Using \S~\ref{sec:52}, one has a
  decomposition of $M$ into skew-Hermitian modules
\begin{equation} \label{eq:55}
   M\simeq M'\oplus L 
\end{equation}
with $M'\simeq R^r\oplus O_E^{s-1}$ as an $R$-module, particularly
$s>0$. 

Suppose $s>0$ and we prove the statement by induction on $s$.
Then $a_1=0$ and by the same argument we have $M\simeq M'\oplus L$ as
above. By the induction hypothesis,  
 the polarization type of $\psi$ on $M^1+M^2$ is
$(1,\dots,1,2,\dots,2)$ with multiplicity $s$ for $1$, and we get
a decomposition of $M$ into skew-Hermitian modules
\[ M\simeq M_1 \oplus L^{\oplus s}, \]
where $M_1$ is a free $R$-submodule of rank $r$ which is self-dual
with respect to the pairing $\psi$. This proves both (1) and (2). \qed
\end{proof}

\subsection{}
\label{sec:54}
Now we classify
self-dual skew-Hermitian modules $(M,\psi)$ in the case where $M\simeq
R^r$, where $r$ is a positive integer. Let $N$ be the
smallest $O_E$-module in $V$ containing $M$, and let $N'=M^1+M^2$. We
have
\[ N=N^1\oplus N^2, \quad N'=2N. \]
The polarization type of $\psi$ on $2N$ is $(2,\dots, 2)$. Put
$\psi_N:=2\psi$. Then we have $\psi_N(N,N)\subset \Z_2$ and 
$N$ is self-dual with respect to the pairing $\psi_N$. 
Put $\ol N=N/2N$ and let $\ol \psi_N:\ol N\times \ol N\to
\F_2$ be the induced non-degenerate pairing. We have 
\begin{equation}
  \label{eq:56}
  2N\subset M\subset N,\quad \dim_{\F_2} \ol M = r
\end{equation}
and that $\ol M$ is isotropic for $\ol \psi_N$, where $\ol M:=M/2N$. Note
that $M$ generates $N$ over $O_E$, or equivalently, $\ol M$ generates
$\ol N$ over $O_E/2=\F_2\times \F_2$. 

Suppose $M_1$ is another self-dual skew-Hermitian module such
that $M_1\simeq R^r$. Define $N_1$ and $\psi_{N_1}$ similarly. 
Since $(N_1,\psi_{N_1})$ is isomorphic to $(N,\psi_N)$ by Lemma~\ref{51}
(2), 
we choose an isomorphism $\alpha: (N_1,\psi_{N_1})\simeq (N,\psi_N)$ of
skew-Hermitian modules. The image $M'=\alpha(M_1)$ is an $R$-module
which satisfies the same property as $M$ does (\ref{eq:56}).

Now we fix the self-dual skew-Hermitian module $(N,\psi_N)$. Consider
the set $X_r$ of $R$-submodules $M$ of $N$ such that 
\begin{itemize}
\item $2N\subset M\subset N$ and $\dim_{\F_2} \ol M = r$, 
\item $\ol M$ is isotropic with respect to the pairing $\ol \psi_N$,
  and
\item $M$ generates $N$ over $O_E$.
\end{itemize}
We need to determine the isomorphism classes of elements $M$ 
in $X_r$. Let $\ol X_r$ be the set of maximally isotropic
$\F_2$-subspaces $\ol M$ of $\ol N$ such that 
$\ol M$ generates $\ol N$ over $\F_2\times
\F_2$. Since the $R$-module structure of $\ol N$ is simply an
$\F_2$-module, the reduction map $M\mapsto \ol M$ gives rise
to a bijection $X_r\simeq \ol X_r$. 


If two elements $M_1$ and $M_2$ in $X_r$ are isomorphic as skew-Hermitian
modules, then any isomorphism between them lifts to an $R$-linear
automorphism of 
$(N,\psi_N)$. Therefore, the set of isomorphism classes of elements in
$X_r$ is in bijection with the quotient 
$\Aut_R(N,\psi_N)\backslash X_r$. As the
$R$-action on $N$ extends uniquely to an $O_E$-action on $N$, we have 
$\Aut_{R}(N,\psi_N)=\Aut_{O_E}(N,\psi_N)$. The action of the group
$\Aut_{O_E}(N,\psi_N)$ on $X_r$ factors through $\Aut_{\ol O_E}(\ol
N,\ol \psi_N)$, and the reduction map yields a bijection 
\[ \Aut_{O_E}(N,\psi_N)\backslash X_r\simeq 
\Aut_{\ol O_E}(\ol N,\ol \psi_N)\backslash \ol X_r. \]


\begin{prop}\label{52}
   Let $r$ be a positive integer.
   There are 
\[ \#\, \Aut_{\ol O_E}(\ol N,\ol \psi_N)\backslash \ol X_r.  \]
non-isomorphic self-dual skew-Hermitian modules $M$ such that $M\simeq
R^r$ as $R$-modules.
\end{prop}

We now describe the set $\Aut_{\ol O_E}(\ol N,\ol \psi_N)\backslash
\ol X_r$. Let $\ol N=\ol N^1\oplus \ol N^2$ be the decomposition
induced by $\ol O_E=\F_2\times \F_2$, and let $p_i:\ol N\to \ol N^i$
be the $i$th projection for $i=1,2$. 
Fix a basis $e_1,\dots, e_r$ for $\ol N^1$ and a basis $f_1,\dots,
f_r$ for $\ol N^2$ such that $\ol \psi_N(e_i,f_j)=\delta_{i,j}$ for all
$i,j=1,\dots, r$. Using the basis $\{e_i, f_i\}_{i=1,\dots, r}$ we
identify the $\F_2$-vector space $\ol N$ with the $\F_2$-space
$\F_2^{2r}$ of column vectors. Let $\ol M$ be an $\F_2$-vector space
of $\ol N$ of dimension $r$. Choose a basis $v_1, \dots, v_r$ for $\ol
M$. The subspace $\ol M$ generates $\ol N$
over $\F_2\times \F_2$ if and only if $p_i(\ol M)=\ol N^i$ for
$i=1,2$. In this case, after a unique change of basis we may assume
that $v_i=e_i+u_i$ for $i=1,\dots,r$ and $u_1,\dots, u_r$ forms a
basis for $\ol N^2$. Write $u_j=\sum_i u_{ij} f_i$ and let
$U:=(u_{ij})\in \GL_r(\F_2)$. The matrix $U$ is uniquely determined by
$\ol M$. One computes
\[ \ol \psi_N(e_i+u_i,e_j+u_j)=\ol \psi_N(e_i,u_j)-\ol
\psi_N(e_j,u_i)=u_{ij}-u_{ji}. \]
It follows that $\ol M$ is isotropic for $\ol \psi_N$ if and only if
$U$ is symmetric. This shows $\ol X_r\simeq S_r$. With respect to the
basis $\{e_i,f_i\}$ the group $\Aut_{\ol O_E}(\ol N,\ol \psi_N)$ of
automorphisms is 
\[ \left \{
  \begin{pmatrix}
    A^{-1} &  0 \\ 0 & A^t 
  \end{pmatrix}\, ;\, A\in \GL_r(\F_2)\, \right\}\simeq \GL_r(\F_2). \]
From 
\[ \begin{pmatrix}
    A^{-1} &  0 \\ 0 & A^t 
  \end{pmatrix}
  \begin{pmatrix}
    I_r \\ U 
  \end{pmatrix}=
  \begin{pmatrix}
    I_r \\ A^t U A
  \end{pmatrix} A^{-1} \]
the action of the group $\GL_r(\F_2)\simeq \Aut_{\ol O_E}(\ol N,\ol
\psi_N)$ on $S_r$, by transport of the action of $\ol X_r$, is given by
\[ A\cdot U=A^t U A,\quad \forall\,A\in \GL_r(\F_2), U\in S_r. \]
We have shown

\begin{prop}\label{53}
  Notation as above. 
  There is a bijection
\[ \Aut_{\ol O_E}(\ol N,\ol \psi_N)\backslash \ol X_r\simeq S_r/\!\sim. \]
\end{prop}

\begin{prop}[Witt Cancellation]\label{54}
  Let $M_1$ and $M_2$ be two self-dual skew-Hermitian modules. Suppose
  there is an isomorphism
\[ M_1\oplus L^{\oplus s}\simeq M_2 \oplus L^{\oplus s} \]
of skew-Hermitian modules for some integer $s\ge 0$. Then $M_1$ is
isometric to $M_2$. 
\end{prop}
\begin{proof}
  By Lemma~\ref{51} (2), we have decompositions $M_1=M_1'\oplus
  L^{\oplus s'}$ and $M_2=M_2'\oplus L^{\oplus s''}$ as skew-Hermitian
  modules such that $M_1'$ and $M_2'$ are free as $R$-modules. Note
  that $s''=s'$.
  Replacing $M_1$ and $M_2$ by $M_1'$ and $M_2'$,
  respectively, we may assume that $M_1\simeq M_2\simeq R^r$ as
  $R$-modules. Write $\psi_i$ for the pairings on $M_i\oplus L^{\oplus
  s}$, for $i=1,2$. Let $N_i$ be the smallest
  $O_E$-module in the vector space $E M_i$ containing $M_i$. Then
  $N_i\oplus L^{\oplus s}$ is the smallest 
  $O_E$-module in the vector space $E(M_i\oplus L^{\oplus s})$
  containing $M_i\oplus L^{\oplus s}$. Put $\psi'_i:=2\psi_i$, which
  is a $\Z_2$-valued skew-Hermitian form on $N_i\oplus L^{\oplus s}$. The
  pairing $\psi_i'$ induces a non-degenerate pairing $\ol {\psi'_i}$ 
  on the $\F_2$-vector space 
\[ (N_i\oplus L^{\oplus s})/(2N_i\oplus L^{\oplus s})\simeq
  N_i/2N_i=:\ol N_i, \]
and the subspace 
\[ (M_i\oplus L^{\oplus s})/(2N_i\oplus L^{\oplus s})\simeq
  M_i/2N_i=:\ol M_i \] 
is a maximal isotropic subspace with respect to $\ol {\psi'_i}$. 

Let $\alpha:M_1\oplus L^{\oplus s}\simeq M_2 \oplus L^{\oplus s}$ be
an isomorphism of skew-Hermitian modules. The map $\alpha$ lifts to an
isomorphism $\beta: N_1\oplus L^{\oplus s}\simeq N_2 \oplus L^{\oplus
  s}$, and $\beta$ induces an isomorphism $\bar \beta: \ol N_1\simeq
\ol N_2$ of symplectic $\F_2$-spaces such that $\bar \beta(\ol
M_1)=\ol M_2$. We lift Lagrangian bases $\{x_i\}$ for $\ol N_1$) and
$\{\bar \beta (x_i)\}$ for $\ol N_2$ to Lagrangian bases $X_i$ for
$N_1$ and $Y_i$ for $N_2$, respectively. The map $\gamma:N_1\simeq
N_2$ which sends $X_i$ to $Y_i$ is an isomorphism of skew-Hermitian
modules. Since $\gamma(2N_1)=2N_2$ and $\gamma(\ol M_1)=\ol M_2$, 
it gives an isomorphism from $M_1$ to $M_2$. 
This proves the proposition. \qed
\end{proof}


\begin{remark}\label{55}
Proposition~\ref{54} (also Proposition~\ref{67}) 
is not covered by a general Witt type
cancellation theorem \cite[Theorem 3]{bayer-fluckiger-fainsilber}
proved by Bayer-Fluckiger and Fainsilber, as the condition $a+\bar a=1$
for some $a\in R$ is not fulfilled.
\end{remark}

\begin{cor}\label{56}
  Let $M$ be a self-dual skew-Hermitian module of $\Z_2$-rank
  $2n$. Then there are unique non-integers $r$ and $s$ with $r+s=n$
  and a self-duel skew-Hermitian module $M_1$ which is free of rank
  $r$ such that
\[ M\simeq M_1\oplus L^{\oplus s}. \]
Moreover, $M_1$ is uniquely determined, up to isomorphism,  by $M$.
\end{cor}
\begin{proof}
  This follows immediately from Lemma~\ref{51} (2) and
  Proposition~\ref{54}. \qed
\end{proof}

By Corollary~\ref{56} and Propositions~\ref{52} and \ref{53},
Theorem~\ref{11} is proved.

\subsection{}
\label{sec:55}
It is well-known that the set $S_r/\!\sim$ parametrizes equivalence
classes of non-degenerate symmetric $\F_2$-spaces $(W,\varphi)$ of
dimension $r$. We write $(1)$ for the non-degenerate one-dimensional
symmetric space $\F_2$ with the pairing $\varphi(e_1,e_1)=1$, and write
\[
\begin{pmatrix}
  0 & 1 \\ 1 & 0
\end{pmatrix} \]
for the non-degenerate two-dimensional symmetric space
$\F_2\oplus \F_2$ with the pairing 
\[ \varphi (e_1,e_1)=\varphi(e_2,e_2)=0, \quad \varphi(e_1,e_2)=1. \]

\begin{lemma}\label{57} 
  Let $(W,\varphi)$ be a symmetric space over $\F_2$ of dimension $r\ge
  1$. \\ If $r$ is odd, then 
\[ (W,\varphi)\simeq (1)\oplus \begin{pmatrix}
  0 & 1 \\ 1 & 0
\end{pmatrix} \oplus \dots \oplus \begin{pmatrix}
  0 & 1 \\ 1 & 0
\end{pmatrix}. \]
If $r$ is even, then 
\begin{equation*}
  \begin{split}
    (W,\varphi)& \simeq  \begin{pmatrix}
  0 & 1 \\ 1 & 0
\end{pmatrix} \oplus \dots \oplus \begin{pmatrix}
  0 & 1 \\ 1 & 0 \\
\end{pmatrix}, \ \text{or} \\
  & \simeq  \begin{pmatrix}
  0 & 1 \\ 1 & 0
\end{pmatrix} \oplus \dots \oplus \begin{pmatrix}
  0 & 1 \\ 1 & 0 \\
\end{pmatrix} \oplus (1)\oplus (1).
  \end{split}
\end{equation*}

In particular, we have
\[ \# S_r/\!\sim=
\begin{cases}
  1, & \text{if $r$ is odd},\\
  2, & \text{if $r$ is even}. 
\end{cases} \]
\end{lemma}
\begin{proof}
  Denote by $N(W)$ the set of vectors $x\in W$ such that
  $\varphi(x,x)=0$. It is easy to show that $N(W)$ is a subspace. 
  Let $N(W)_n$ be the null subspace of $N(W)$ and $W_2\subset N(W)$ be
  a complement of $N(W)_n$. Then $\varphi$ is non-degenerate on the
  subspace $W_2$ and we have 
\[ W_2\simeq \begin{pmatrix}
  0 & 1 \\ 1 & 0
\end{pmatrix} \oplus \dots \oplus \begin{pmatrix}
  0 & 1 \\ 1 & 0 \\
\end{pmatrix}. \]
Let $W_1:=W_2^{\bot}\subset W$ be the orthogonal complement of
$W_2$; one has $W=W_1\oplus W_2$. We have $N(W_1)=N(W)_n\subset W_1$. 
If $\dim N(W_1)\ge 2$, then $\varphi$
is degenerate on $W_1)$. Therefore, $\dim W_1\le 2$. As there are no
vectors $x,y\in W_1$ such that $\varphi(x,x)=\varphi(y,y)=0$ and
$\varphi(x,y)=1$, the symmetric space $(W_1,\varphi)$ is isomorphic to
the direct sum of copies of $(1)$. Therefore, the lemma
follows.\qed 
\end{proof}

\section{Proof of Theorem~\ref{12}.}
\label{sec:06}


\subsection{}
\label{sec:61}
 
In this section we assume that $p \equiv 3 \ (\! \mod 8)$. Thus, the
algebra $E$ is a unramified quadratic field extension of $\Q_2$. 
Recall that 
\[ R=\Z_2[\omega]=\Z_2[X]/(X^2+2X+p+1)\subset
O_E=\Z_2[\alpha]/(X^2+X+(p+1)/4), \]
$\omega=\pi-1$, and $\pi\in R$ is an element with $\pi^2=-p$. 
Let $(M,\psi)$ be a 
skew-Hermitian module over $R$ with $\Z_2$-rank $2n$, and
let $V:=M\otimes_{\Z_2} \Q_2$. There is a unique non-degenerate $E$-valued
skew-Hermitian form
\begin{equation}
  \label{eq:61}
  \<\, ,\, \>:V\times V\to E  
\end{equation}
such that $\<ax,by\>=a \bar b \< x,y \>$ for all $a, b\in E$ and $x,y\in
V$,  and $\psi(x,y)=\Tr_{E/\Q_2}\<x,y\>$ for all $x, y \in V$. Since
$E/\Q_2$ is unramified, the inverse different $\scrD_{E/\Q_2}^{-1}$ is
equal to $O_E$. Put $(\, ,\,):=\pi \<\, ,\,\>$, which is a Hermitian
form on $V$. Note that $\<M, M\>\subset O_E$ if and only if
$(M,M)\subset O_E$, as the element $\pi=1+\omega$ lands in
$1+2O_E=R^\times \subset O_E^\times$. 

\begin{lemma}\label{61} \
\begin{enumerate}
  \item One has $(M,M)\subset 2^{-1}R$. Under the assumption that $M$
    is self-dual with respect to the pairing $\psi$, 
    the condition $(M,M)\subset
    O_E$ holds if and only if $M$ is invariant under the $O_E$-action.

\item The $R$-lattice $M$ is self-dual with respect to the pairing
 $\psi$ if and only if $M$ is self-dual with resspt to the pairing
 $2(\,,\,)$. 
\end{enumerate}
\end{lemma}
\begin{proof}
  (1) The element $\<x,y\>$ for $x, y\in M$ satisfies the property
      $\Tr_{E/\Q_2} (\<x,y\>R)\subset \Z_2$. Therefore, it lands in
      the dual lattice $R^\vee$ of $R$ for the pairing $(a,b)\mapsto
      \Tr_{E/\Q_2} (ab)$. It is easy to check that $R^\vee=2^{-1}R$, and
      hence the first part is proved. 
 
      Let $\wt M$ be the $O_E$-submodule in $V$ generated by $M$. If
      $M$ is invariant under the $O_E$-action, then $\<M,M\>\subset
      O_E^\vee=O_E$. Conversely, if $\<M,M\>\subset O_E$, then $\<\wt
      M,\wt M\>\subset O_E$. This yields $\psi(\wt M,\wt M)\subset
      \Z_2$. As $M$ is self-dual, one has $\wt M=M$. This proves (1).   

   (2) For any element $x\in V$, one has
\[ (x,M)\in {2}^{-1} R \iff \psi(x,M)\in \Z_2. \]
Therefore, the assertion follows. \qed
\end{proof}

From now on, we assume that $M$ is self-dual with respect to the
pairing $\psi$. 

\begin{lemma}\label{62}
  One has $N_{E/\Q_2}(R^\times)=\Z_2^\times$.
\end{lemma}
\begin{proof}
  We use the following basic fact in number theory; see
  \cite[Corollary, p.~7]{platonov-rapinchuk:agnt}. 
  Suppose $E/F$ is a 
  unramified finite extension of non-Archimedean local fields. For any
  integer $i\ge 1$, let
  $U_E^{(i)}:=1+\pi_E^i O_E \subset O_E^\times$ and
  $U_F^{(i)}:=1+\pi_F^i 
  O_F \subset O_F^\times$ be the $i$th principal 
  congruence subgroups
  of $O_E^\times$ and $O_F^\times$, respectively. Then we have
  $N_{E/F}(U^{(i)}_E)=U^{(i)}_F$.

  Since $R^\times=1+2O_E$ and $\Z_2^\times=1+2\Z_2$, the lemma
  follows.\qed         
\end{proof}

\subsection{}
We define three self-dual skew-Hermitian modules $L_0,L_1, H$ as
follows. 

\begin{itemize}
\item[(i)] The $R$-module $L_0$ is $O_E$ and the pairing $\psi_0$ is
  defined by $\psi_0(1,\omega/2)=1$. Write $\omega'$ for the conjugate
  of $\omega$. One checks that  
\begin{equation}
    \label{eq:62}
   \psi_0(\omega'/2,1)=\psi_0((-2-\omega)/2,1)
    =\psi_0(-\omega/2,1)=\psi_0(1,\omega/2). 
\end{equation}
\item[(ii)] The $R$-module $L_1$ is $R$ and the pairing $\psi_1$ is
  defined by $\psi_1(1,\omega)=1$. One checks that 
  $\psi_1(1,\omega)=\psi_1(\omega',1)$ as (\ref{eq:62}).
\item[(iii)] The $R$-module $H$ is $R\oplus R$ with standard basis
  $e_1=(1,0)$ and $e_2=(0,1)$. The pairing $\psi_H$ is defined by 
  \begin{equation}
    \label{eq:63}
  \begin{split}
\psi_H(e_1,\omega e_1)& =\psi_H(e_1,e_2)=\psi_H(e_2,\omega
    e_2)=\psi_H(\omega e_1,\omega e_2)=0,\\
&  \psi_H(e_1,\omega
    e_2)=\psi_H(e_2, \omega e_1)=1.       
  \end{split}
  \end{equation}
Note that using the relation $\psi_H(\omega x,y)=\psi_H(x, \omega'
y)$, the pairing $\psi_H$ is uniquely determined by any values of 
\[ \psi_H(e_1, \ \omega e_1), \ \psi_H(e_1, e_2), \ \psi_H(e_1,\omega e_2),
\ \text{and} \ \psi_H(e_2,\omega e_2). \]
One checks that  
\[ \psi_H(\omega'
e_1,e_2)=\psi_H((-2-\omega)e_1,e_2)=\psi_H(e_2,\omega
e_1)=1=\psi(e_1,\omega e_2). \] 
\end{itemize}

By Theorem~\ref{41}, there are unique non-negative integers $r$ and
$s$ with $r+s=n$ 
such that \[ M\simeq R^r\oplus O_E^s\] 
as $R$-modules. If one has a
decomposition  
\[ M\simeq L_1^{\oplus r_1}\oplus H^{\oplus r_2} \oplus L_0^{\oplus s'} \]
as skew-Hermitian modules, then $r_1+2r_2=r$ and $s'=s$.

\subsection{}
Let $M_0\subset M$ be the $R$-submodule defined by 
\[ M_0:=\{x\in M\, |\, \omega x\in 2M \, \}. \]
Write $\ol M:=M/2M$ and $\ol M_0:=M_0/2M$. Let 
\[ \ol \psi: \ol M\times \ol M \to \F_2 \]
be the induced non-degenerate alternating pairing. We have a
filtration
\[ \omega \ol M \subset \ol M_0\subset \ol M, \]
and have the following properties
\begin{itemize}
\item $\dim_{\F_2} \omega \ol M=\dim_{\F_2} \ol M/\ol M_0=r$ and
  $\dim_{\F_2} \ol M_0=2s+r$, and
\item $\omega \ol M$ is an isotropic subspace with respect to the
  pairing $\ol \psi$.  
\end{itemize}

For any element $x$ in $M$, write $\ol x$ for the image of $x$ in $\ol
M$.  

\begin{lemma} \label{63} Let $M$ be a self-dual skew-Hermitian module. 
\begin{enumerate}
\item  If there exists an element $x\in M$ such that $\ol \psi(\ol
  x, \omega \ol x)\neq 0$, then the $R$-submodule $M_1$ generated by
  $x$, with the pairing $\psi|_{M_1}$ restricted on $M_1$, is
  isomorphic to $(L_1,\psi_1)$ as skew-Hermitian modules. Moreover, we
  have the decomposition
\[ M=M_1\oplus M_1^{\bot} \]
as skew-Hermitian modules. 
\item If there exists an element $x\in M_0\subset M$ such that $\ol \psi(\ol
  x, \ol {\omega/2  x})\neq 0$, then the $R$-submodule $M_1$ generated by
  $x$ and $\omega/2 x$, with the pairing $\psi|_{M_1}$ restricted on
  $M_1$, is 
  isomorphic to $(L_0,\psi_0)$ as skew-Hermitian modules. Moreover, we
  have the decomposition
\[ M=M_1\oplus M_1^{\bot} \]
as skew-Hermitian modules. 

\end{enumerate}
\end{lemma}
\begin{proof}
  (1) We have $\psi(x,\omega x)=a\in \Z^{\times}_2$. By
      Lemma~\ref{62}, there is an element $\lambda\in R^\times$ such
      that $N(\lambda)=a^{-1}$. Replacing $x$ by $\lambda x$, we get
      $\psi(x,\omega x)=1$. Since $M_1$ is self-dual and
      $R$-invariant, the orthogonal complement $M_1^{\bot}$ is $R$-invariant
      and we have $M=M_1\oplus M_1^{\bot}$. This proves (1).

  (2) Using the same argument as (1), we can choose an element $x\in
      M_1$ such that $\psi(x,\omega/2 x)=1$ and we have the
      decomposition $M=M_1\oplus M_1^{\bot}$. It is clear that we have
      an isometry $(M_1,\psi|_{M_1})\simeq (L_0,\psi_0)$. This proves
      (2). \qed 
\end{proof}

\begin{lemma}\label{64}
  There is a decomposition as skew-Hermitian modules
  \begin{equation}
    \label{eq:64}
    M\simeq M_1\oplus L_0^{\oplus s}
  \end{equation}
where $M_1$ is a free $R$-submodule of rank $r$ which is self-dual
with respect to the pairing $\psi$. 
\end{lemma}
\begin{proof}
  It is clear that $\ol \psi(\omega \ol M, \ol M_0)=0$. It follows from
the non-degeneracy of $\ol \psi$ 
that there are vectors $x_1,\dots, x_r\in \omega \ol M$ and $y_1, \dots
, y_r\in \ol M$ such that 
\[ \ol \psi(x_i,x_j)=\ol \psi(y_i,y_j)=0,\ 
\text{ and }\  \ol \psi(x_i, y_j)=\delta_{i,j} \] 
for all $i,j=1,\dots, r$. Put
$W_1:=<y_1,\dots, y_r>_{\F_2}$ and $W_2:=(\omega \ol M\oplus
W_1)^{\bot}$. We have $\ol M=\omega \ol M\oplus W_1\oplus W_2$.
It follows from 
\[ \ol \psi(\omega W_2,\ol M)=\ol \psi (W_2, (-2-\omega) \ol M)=0 \]
that $\omega W_2=0$, and hence $W_2 \subset \ol M_0$. It follows from the
dimension count that $\ol M_0=\omega \ol M\oplus W_2$.
This shows that the pairing $\ol \psi$ is non-degenerate on $\ol
M_0/\omega \ol M$. We can choose elements $z_1,\dots, z_s$ in $M_0$
such that the $O_E$-submodule $M_2=<z_1,\dots, z_s>_{O_E}\subset M_0$
is mapped onto $W_2$. The module $M_2$ is self-dual with respect to
$\psi$. By Lemma~\ref{61}, 
$M_2$ is a unimodular Hermitian module over $O_E$ for
the pairing $(\, , \,)$, and hence $M_2$ is a direct sum of $O_E$-rank
one Hermitian submodule, as $E/\Q_2$ is unramified. 
Therefore, $M_2\simeq L_0^{\oplus s}$ as skew-Hermitian modules. 
Put $M_1:=M_2^\bot$. We have $M=M_1\oplus M_2$ as skew-Hermitian
modules, and $M_1$ is a free $R$-submodule of rank $r$ which is
self-dual with respect to the pairing $\psi$.  \qed  
\end{proof}

\subsection{}
\label{}
Now we classify
self-dual skew-Hermitian modules $(M,\psi)$ in the case where $M\simeq
R^r$, where $r$ is a positive integer. Let $N$ be the
smallest $O_E$-module in $V$ containing $M$, and let $N'$ be the
largest $O_E$-submodule in $M$. We
have $N'=2N\supset 2M$. Since $N'/2M$ is a maximal isotropic
$\F_2$-subspace of dimension $r$
with respect to $\ol \psi$, the polarization type 
of $\psi$ on $N'$ is $(2,\dots, 2)$. Put
$\psi_N:=2\psi$. We have $\psi_N(N,N)\subset \Z_2$, and $N$ is
self-dual with respect to the pairing $\psi_N$. 
Note that $(N,\psi_N)\simeq (L_0,\psi_0)^{\oplus r}$ by Lemma~\ref{64}. 
Put $\ol N=N/2N$ and let $\ol \psi_N:\ol N\times \ol N\to
\F_2$ be the induced non-degenerate pairing. We have 
\begin{equation}
  \label{eq:65}
  2N\subset M\subset N,\quad \dim_{\F_2} \ol M = r
\end{equation}
and $\ol M$ is isotropic for $\ol \psi_N$, where $\ol M:=M/2N$. Note
that $M$ generates $N$ over $O_E$, or equivalently, the subspace $\ol
M$ generates $\ol N$ over $\ol O_E:=O_E/2O_E=\F_4$. 

Suppose $M_1$ is another self-dual skew-Hermitian module such
that $M_1\simeq R^r$. Define $N_1$ and $\psi_{N_1}$ similarly. As
$(N_1,\psi_{N_1})\simeq (N,\psi_N)$, we choose 
an isomorphism $\alpha: (N_1,\psi_{N_1})\simeq (N,\psi_N)$ of
skew-Hermitian modules. The image $M'=\alpha(M_1)$ is an $R$-module
which satisfies the same property as $M$ does (\ref{eq:65}).

Now we fix the self-dual skew-Hermitian module $(N,\psi_N)$. 
Similarly to \S~\ref{sec:54}, consider
the set $X_r$ of $R$-submodules $M$ of $N$
such that 
\begin{itemize}
\item $2N\subset M\subset N$ and $\dim_{\F_2} \ol M = r$, 
\item $\ol M$ is isotropic with respect to the pairing $\ol \psi_N$,
  and
\item $M$ generates $N$ over $O_E$.
\end{itemize}
We need to determine the isomorphism classes of elements $M$ 
in $X_r$. Let $\ol X_r$ be the set of maximally isotropic
$\F_2$-subspaces $\ol M$ of $\ol N$ such that 
$\ol M$ generates $\ol N$ over $\F_4$. 
Since the $R$-module structure of $\ol N$ is simply an
$\F_2$-module, the reduction map $M\mapsto \ol M$ gives rise
to a bijection $X_r\simeq \ol X_r$. 

Using the same argument as in \S~\ref{sec:54}, the set of isomorphism
classes of elements in $X_r$ is in bijection with
$\Aut_{O_E}(N,\psi_N)\backslash X_r$, and the reduction map induces
the bijection 
\begin{equation}
  \label{eq:66}
  \Aut_{O_E}(N,\psi_N)\backslash X_r\simeq 
  \Aut_{\F_4}(\ol N,\ol \psi_N)\backslash \ol X_r
\end{equation}




\begin{prop}\label{65}
  Let $r$ be a positive integer.
  There are 
\[ \# \Aut_{\F_4}(\ol N,\ol \psi_N)\backslash \ol X_r \]
non-isomorphic self-dual skew-Hermitian modules $M$ such that $M\simeq
R^r$ as $R$-modules.
\end{prop}
 
We now describe the set $\Aut_{\ol O_E}(\ol N,\ol \psi_N)\backslash
\ol X_r$. Let $V_0=\F_4^r$, viewed as the space of column vectors,  
together with the standard non-degenerate
Hermitian form $(\, , )_0$ defined by 
\[ ((x_i), (y_i))_0=\sum_{i=1}^r x_i \bar y_i, \]
where $y\mapsto \bar y$ is the non-trivial automorphism on
$\F_4$. Choose an element $\epsilon\in F_4^\times$ such that
$\epsilon+\bar \epsilon=0$ (in fact $\epsilon=1$) and put 
$\<x,y\>_0:=\Tr_{\F_4/\F_2} \epsilon(x,y)_0$, for $x,y\in V_0$. 
Then we have an isomorphism $(V_0,\<\,,\,\>_0)\simeq (\ol N,\ol
\psi_N)$ as skew-Hermitian modules over $\F_4$ and we may 
assume that $(\ol N,\ol \psi_N)=(V_0,\<\,,\,\>_0)$. Let 
\[ U(r)(\F_2)=\Aut_{\F_4}(V_0,(\,,\,)_0)=\Aut_{\F_4}(V_0,\<\,,\,\>_0)
\]
be the unitary group associated to $(V_0,(\,,\,)_0)$. 
Write any matrix $U$ in $\GL_r(\F_4)$ as $(u_1,\dots, u_r)$, where
$u_i\in V_0$. The map 
\[ U=(u_1,\dots, u_r)\mapsto W=<u_1,\dots,u_r>_{\F_2} \]
induces a bijection between $\GL_r(\F_4)/\GL_r(\F_2)$ and the set of
$r$-dimensional $\F_2$-subspace $W\subset V_0$ which spans $V_0$ over
$\F_4$. Note that $\Tr_{\F_4/\F_2}(x,y)_0=0$ if and only if
$(x,y)_0\in \F_2$. It follows that $W$ is an isotropic subspace with
respect to $\<\,,\,\>_0$ if and only if $\bar U^t U\in \GL_r(\F_2)$
for any  matrix $U$ mapping to $W$. Let $Y_r\subset \GL_r(\F_4)$ be
the subset consisting of matrices $U$ such that $\bar U^t U\in
\GL_r(\F_2)$. The action of the group $U(r)(\F_2)$ on the set
$\GL_r(\F_4)/\GL_r(\F_2)$, by transport of the structure, is simply the
left translation. Thus, there is a bijection
\begin{equation}
  \label{eq:67}
  \Aut_{\ol O_E}(\ol N,\ol \psi_N)\backslash
\ol X_r\simeq U(r)(\F_2)\backslash  Y_r /\GL_r(\F_2).
\end{equation}

\begin{lemma}\label{66}
  The map $\pi: Y_r \to \GL_r(\F_2)$, $U\mapsto \bar U^t U$, induces
  the bijection
\[ U(r)(\F_2)\backslash  Y_r /\GL_r(\F_2)\simeq S_r/\!\sim. \]
\end{lemma}
\begin{proof}
  Since the matrix $\bar U^t U$ is Hermitian, one has
  $\pi(Y_r)\subset S_r$. Each non-empty fiber of $\pi$ is a principal
  homogeneous space of $U(r)(\F_2)$. Thus, the induced map
  $\pi:U(r)(\F_2) \backslash Y_r\to S_r$ is injective. One easily checks 
  $\pi(U P)=P^t \pi(U) P$ for $P\in \GL_r(\F_2)$. Thus, we have
  the injection map 
\[  U(r)(\F_2)\backslash  Y_r /\GL_r(\F_2)\to S_r/\!\sim. \]
It suffices to show that the map $\pi:Y_r \to S_r/\!\sim$ is
surjective. By Lemma~\ref{57}, any matrix $A$ in $S_r$ is equivalent
to a matrix with diagonal boxes either $(1)$ or $\wt I_2=
\begin{pmatrix}
  0 & 1 \\ 1 & 0 \\
\end{pmatrix}$. Thus it suffices to find a matrix $U\in \GL_2(\F_4)$
such that $\bar U^t U =\wt I_2$. Note that $\F_4=\F_2[\alpha]$ with
$\alpha^2+ \alpha+1=0$. Take 
\[ U=
\begin{pmatrix}
  \alpha & \alpha+1 \\ 1 & \alpha+1
\end{pmatrix} \]
and get $\bar U^t U=\wt I_2$. This proves the lemma. \qed
\end{proof}

\begin{prop}[Witt Cancellation]\label{67}
  Let $M_1$ and $M_2$ be two self-dual skew-Hermitian modules. Suppose
  there is an isomorphism
\[ M_1\oplus L^{\oplus s}\simeq M_2 \oplus L^{\oplus s} \]
of skew-Hermitian modules for some integer $s\ge 0$. Then $M_1$ is
isometric to $M_2$. 
\end{prop}
\begin{proof}
  The proof is the same as that of Proposition~\ref{54}. Note that
  one can lift the isomorphism $\ol \beta:(\ol N_1, \ol \psi_1)\simeq
  (\ol N_2,\ol \psi_2)$ to an isomorphism $\gamma:(N_1,\psi_1)\simeq
  (N_2,\psi_2)$. This is because
  $\Isom_{R}((N_1,\psi_1),
  (N_2,\psi_2))$ is the set of $\Z_2$-points of a smooth scheme over
  $\Z_2$ (in fact a trivial torsor of a smooth group scheme over 
  $\Z_2$). \qed 
\end{proof}

\begin{cor}\label{68}
  Let $M$ be a self-dual skew-Hermitian module of $\Z_2$-rank
  $2n$. Then there are unique non-integers $r$ and $s$ with $r+s=n$
  and a self-duel skew-Hermitian module $M_1$ which is free of rank
  $r$ such that
\[ M\simeq M_1\oplus L^{\oplus s}. \]
Moreover, $M_1$ is uniquely determined, up to isomorphism,  by $M$.
\end{cor}
\begin{proof}
  This follows immediately from Lemma~\ref{64} and
  Proposition~\ref{67}. \qed
\end{proof}

Theorem~\ref{12} follows from
Corollary~\ref{68}, Proposition~\ref{65}, (\ref{eq:67}), and
Lemma~\ref{66}.


\begin{thebibliography}{10}

\def\jams{{\it J. Amer. Math. Soc.}} 
\def\invent{{\it Invent. Math.}} 
\def\ann{{\it Ann. Math.}} 
\def\ihes{{\it Inst. Hautes \'Etudes Sci. Publ. Math.}} 

\def\ecole{{\it Ann. Sci. \'Ecole Norm. Sup.}}
\def\ecole4{{\it Ann. Sci. \'Ecole Norm. Sup. (4)}}
\def\mathann{{\it Math. Ann.}} 
\def\duke{{\it Duke Math. J.}} 
\def\jag{{\it J. Algebraic Geom.}} 
\def\advmath{{\it Adv. Math.}}
\def\compos{{\it Compositio Math.}} 
\def\ajm{{\it Amer. J. Math.}}
\def\grenoble{{\it Ann. Inst. Fourier (Grenoble)}}
\def\crelle{{\it J. Reine Angew. Math.}}
\def\mrt{{\it Math. Res. Lett.}}
\def\imrn{{\it Int. Math. Res. Not.}}
\def\acad{{\it Proc. Nat. Acad. Sci. USA}}
\def\tams{{\it Trans. Amer. Math. Sci.}}
\def\cras{{\it C. R. Acad. Sci. Paris S\'er. I Math.}} 
\def\mathz{{\it Math. Z.}} 
\def\cmh{{\it Comment. Math. Helv.}}
\def\docmath{{\it Doc. Math. }}
\def\asian{{\it Asian J. Math.}}
\def\jussieu{{\it J. Inst. Math. Jussieu}}

\def\manmath{{\it Manuscripta Math.}} 
\def\jnt{{\it J. Number Theory}} 
\def\ijm{{\it Israel J. Math.}}
\def\ja{{\it J. Algebra}} 
\def\pams{{\it Proc. Amer. Math. Sci.}}
\def\smfmemoir{{\it Bull. Soc. Math. France, Memoire}}
\def\bsmf{{\it Bull. Soc. Math. France}}
\def\sb{{\it S\'em. Bourbaki Exp.}}
\def\jpaa{{\it J. Pure Appl. Algebra}}
\def\jems{{\it J. Eur. Math. Soc. (JEMS)}}
\def\jtokyo{{\it J. Fac. Sci. Univ. Tokyo}}
\def\cjm{{\it Canad. J. Math.}}
\def\jaums{{\it J. Australian Math. Soc.}}
\def\pspm{{\it Proc. Symp. Pure. Math.}}
\def\ast{{\it Ast\'eriques}}
\def\pamq{{\it Pure Appl. Math. Q.}}
\def\nagoya{{\it Nagoya Math. J.}}
\def\forum{{\it Forum Math. }}

\def\tp{{to appear}}

\newcommand{\princeton}[1]{Ann. Math. Studies #1, Princeton
  Univ. Press}

\newcommand{\LNM}[1]{Lecture Notes in Math., vol. #1, Springer-Verlag}

\bibitem{bayer-fluckiger-fainsilber} E.~Bayer-Fluckiger and
  L.~Fainsilber, Non-unimodular Hermitian
  forms. \invent~{\bf 123} (1996), no. 2, 233--240. 



\bibitem{curtis-reiner:1}
 C.~W.~Curtis and I.~Reiner, {\it  Methods of representation
 theory. Vol. I. With applications to finite groups and orders}. 
 Pure and Applied Mathematics. A Wiley-Interscience Publication. John
 Wiley \& Sons, Inc., New York, 1981, 819 pp.  




\bibitem{fainsiber-morales}
 L.~Fainsilber and J.~Morales, An injectivity result for Hermitian
 forms over local orders. {\it Illinois J. Math.}~{\bf 43} (1999), no. 2,
 391--402.  




\bibitem{ibukiyama-katsura} T. Ibukiyama and T. Katsura, On
  the field of definition of superspecial polarized abelian varieties
  and type numbers. \compos~{\bf 91} (1994), 37--46.   

\bibitem{ibukiyama-katsura-oort} T. Ibukiyama, T. Katsura and F. Oort,
  Supersingular curves of genus two and class numbers.
  \compos~{\bf 57} (1986), 127--152.

\bibitem{jacobowitz:hermitian} R.~Jacobowitz, Hermitian forms over
   local fields. \ajm~{\bf 84} (1962), 441--465. 



\bibitem{kmrt} M.-A. Knus, A. Merkurjev, M. Rost, and J.-P. Tignol,
  The book of involutions. American Mathematical Society, Colloquium
  Publications, 44. {\it Amer. Math. Soc.}, 1998, 593 pp.  


\bibitem{mumford:av} D. Mumford, {\it Abelian Varieties.} Oxford
   University Press, 1974.




\bibitem{platonov-rapinchuk:agnt} V.~Platonov and A.~Rapinchuk,
Algebraic groups and number theory.  
{\it Pure and Applied Mathematics, 139.} 
Academic Press, Inc., Boston, MA, 1994. 

\bibitem{reiner:mo} I.~Reiner, {\it Maximal orders}. 
  London Mathematical Society
  Monographs, No.~{\bf 5}.~Academic Press, London-New York, 1975. 395 pp. 

\bibitem{riehm:hermitian} C.~Riehm, Hermitian forms over local
  hereditary orders. \ajm~{\bf 106} (1984), no. 4, 781--800.

\bibitem{scharlau:hf} W.~Scharlau, {\it Quadratic and Hermitian
  forms}. Grundlehren der Mathematischen Wissenschaften 
  {\bf 270}. Springer-Verlag, Berlin, 1985. 

\bibitem{shimura:hermitian2008}
 G.~Shimura, Arithmetic of Hermitian forms. \docmath~{\bf 13} (2008),
 739--774.  


\bibitem{tate:ht} J. Tate, Classes d'isogenie de vari\'et\'es
  ab\'eliennes sur un corps fini (d'apr\`es T. Honda). \sb~{352}
  (1968/69). \LNM{179}, 1971.

\bibitem{tate:eav} J. Tate, Endomorphisms of abelian varieties over
  finite fields. \invent~{\bf 2} (1966), 134--144.


\bibitem{waterhouse:thesis} W.~C.~Waterhouse, Abelian
  varieties over finite fields. \ecole4~{\bf 2} (1969), 521--560.  









\bibitem{yu:sp_prime}
  C.-F. Yu, Superspecial abelian varieties over finite prime
  fields. arXiv:1004.0120, 12 pp. 




\end{thebibliography}
\end{document}